\documentclass{amsart}
\usepackage{amssymb, amsmath, mathrsfs, verbatim, url, bbm, bbold}

\makeatletter
\@namedef{subjclassname@2020}{%
  \textup{2020} Mathematics Subject Classification}
\makeatother

\theoremstyle{plain}
\newtheorem{theorem}{Theorem}[section]
\newtheorem{lemma}[theorem]{Lemma}
\newtheorem{proposition}[theorem]{Proposition}
\newtheorem{corollary}[theorem]{Corollary}

\theoremstyle{definition}

\newtheorem{question}[theorem]{Question}

\newcommand{\re}{\upharpoonright}
\newcommand{\cl}{\mathsf{cl}}
\newcommand{\dom}{\mathsf{dom}}
\newcommand{\mini}{\mathsf{min}}
\newcommand{\maxi}{\mathsf{max}}
\newcommand{\Exp}{\mathsf{Exp}}
\newcommand{\rank}{\mathsf{rank}}
\newcommand{\kernel}{\mathsf{ker}}
\newcommand{\kc}{\mathsf{kc}}
\newcommand{\kcs}{\mathsf{kc}^\ast}
\newcommand{\weight}{\mathsf{w}}
\newcommand{\zero}{\mathbb{0}}
\newcommand{\one}{\mathbb{1}}
\newcommand{\down}{\!\downarrow\,}
\newcommand{\RO}{\mathsf{RO}}
\newcommand{\SFIP}{\mathsf{SFIP}}
\newcommand{\Aa}{\mathcal{A}}
\newcommand{\BB}{\mathcal{B}}
\newcommand{\CC}{\mathcal{C}}
\newcommand{\DD}{\mathcal{D}}
\newcommand{\FF}{\mathcal{F}}
\newcommand{\GG}{\mathcal{G}}
\newcommand{\KK}{\mathcal{K}}
\newcommand{\OO}{\mathcal{O}}
\newcommand{\PP}{\mathcal{P}}
\newcommand{\UU}{\mathcal{U}}
\newcommand{\VV}{\mathcal{V}}
\newcommand{\BBB}{\mathbb{B}}
\newcommand{\PPP}{\mathbb{P}}
\newcommand{\QQQ}{\mathbb{Q}}
\newcommand{\RRR}{\mathbb{R}}
\newcommand{\cccc}{\mathfrak{c}}
\newcommand{\dddd}{\mathfrak{d}}
\newcommand{\pppp}{\mathfrak{p}}

\begin{document}

\title[Countable spaces and realcompactness]{Countable spaces, realcompactness, and the pseudointersection number}

\author{Claudio Agostini}
\address{Institut f\"{u}r Diskrete Mathematik und Geometrie
\newline\indent Technische Universit\"{a}t Wien
\newline\indent  Wiedner Hauptstra\ss e 8–-10/104
\newline\indent 1040 Vienna, Austria}
\email{claudio.agostini@tuwien.ac.at}
\urladdr{http://www.dmg.tuwien.ac.at/agostini/}

\author{Andrea Medini}
\address{Institut f\"{u}r Diskrete Mathematik und Geometrie
\newline\indent Technische Universit\"{a}t Wien
\newline\indent  Wiedner Hauptstra\ss e 8–-10/104
\newline\indent 1040 Vienna, Austria}
\email{andrea.medini@tuwien.ac.at}
\urladdr{http://www.dmg.tuwien.ac.at/medini/}

\author{Lyubomyr Zdomskyy}
\address{Institut f\"{u}r Diskrete Mathematik und Geometrie
\newline\indent Technische Universit\"{a}t Wien
\newline\indent  Wiedner Hauptstra\ss e 8–-10/104
\newline\indent 1040 Vienna, Austria}
\email{lyubomyr.zdomskyy@tuwien.ac.at}
\urladdr{http://www.dmg.tuwien.ac.at/zdomskyy/}

\date{October 26, 2023}

\begin{abstract}
All spaces are assumed to be Tychonoff. Given a realcompact space $X$, we denote by $\Exp(X)$ the smallest infinite cardinal $\kappa$ such that $X$ is homeomorphic to a closed subspace of $\RRR^\kappa$. Our main result shows that, given a cardinal $\kappa$, the following conditions are equivalent:
\begin{itemize}
\item There exists a countable crowded space $X$ such that $\Exp(X)=\kappa$,
\item $\pppp\leq\kappa\leq\cccc$.
\end{itemize}
In fact, in the case $\dddd\leq\kappa\leq\cccc$, every countable dense subspace of $2^\kappa$ provides such an example. This will follow from our analysis of the pseudocharacter of countable subsets of products of first-countable spaces. Finally, we show that a scattered space of weight $\kappa$ has pseudocharacter at most $\kappa$ in any compactification. This will allow us to calculate $\Exp(X)$ for an arbitrary (that is, not necessarily crowded) countable space.
\end{abstract}



\maketitle

\tableofcontents

\section{Introduction}\label{section_introduction}

By \emph{space} we will mean Tychonoff topological space, and by \emph{countable} we will mean finite or countably infinite. Recall that a space is \emph{realcompact} if it is homeomorphic to a closed subspace of $\RRR^\kappa$ for some cardinal $\kappa$. Given a realcompact space $X$, we will denote by $\Exp(X)$ the smallest infinite cardinal $\kappa$ such that $X$ is homeomorphic to a closed subspace of $\RRR^\kappa$. We will call $\Exp(X)$ the \emph{realcompactness number} of $X$.

It is well-known that Lindel\"{o}f spaces are realcompact (see Theorem \ref{theorem_Exp} or \cite[Theorem 3.11.12]{engelking}). As pointed out by van Douwen (see \cite[Exercise 8.23]{van_douwen}), for every infinite cardinal $\kappa\leq\cccc$ there exists a separable metrizable crowded space $X$ such that $\Exp(X)=\kappa$.\footnote{\,Recall that a \emph{Bernstein set} in $2^\omega$ is a set $B\subseteq 2^\omega$ such that every compact subset of $B$ is countable and every compact subset of $2^\omega\setminus B$ is countable. Fix a Bernstein set $B$ in $2^\omega$, and let $X$ be such that $B\subseteq X\subseteq 2^\omega$ and $|2^\omega\setminus X|=\kappa$. Using Theorem \ref{theorem_Exp}, it is straightforward to check that $\Exp(X)=\kappa$.} However, the situation turns out to be more interesting for countable crowded spaces. More specifically, the following is our main result, which exhibits an unexpected connection between realcompactness and the pseudointersection number $\pppp$.

\begin{theorem}\label{theorem_main}
Given a cardinal $\kappa$, the following are equivalent:
\begin{itemize}
\item There exists a countable crowded space $X$ such that $\Exp(X)=\kappa$,
\item $\pppp\leq\kappa\leq\cccc$.	
\end{itemize}
\end{theorem}
\begin{proof}
This follows from Corollary \ref{corollary_lower_bound}, Theorem \ref{theorem_example_general} and Corollary \ref{corollary_upper_bound}.
\end{proof}

After the preliminaries discussed in Sections \ref{section_preliminaries} and \ref{section_calculating}, we will devote the entirety of Sections \ref{section_lower}, \ref{section_p} and \ref{section_bumping} to the proof of the above theorem. In Section \ref{section_scattered}, we will show that a scattered space of weight $\kappa$ has pseudocharacter at most $\kappa$ in any compactification (see Theorem \ref{theorem_scattered}). While this result seems to be of independent interest, it will allow us to calculate the realcompactness number of an arbitrary (that is, not necessarily crowded) countable space (see Corollary \ref{corollary_hereditarily_lindelof}). In Section~\ref{section_pseudocharacter}, we will give an analysis of the pseudocharacter of countable subsets of products of first-countable spaces, which will allow us to obtain the upper bound for the realcompactness number of a countable space (see Corollary \ref{corollary_upper_bound}). For this analysis, the dominating number $\dddd$ will be relevant, ultimately because of Hechler's discovery that $\Exp(\QQQ)=\dddd$ (see Corollary \ref{corollary_hechler}). In particular, when $\dddd\leq\kappa\leq\cccc$, it will follow that $\Exp(X)=\kappa$ for every countable dense subspace of $2^\kappa$ (see Corollary \ref{corollary_dense_subspace}).

\section{Preliminaries and terminology}\label{section_preliminaries}

Our reference for topological notions is \cite{engelking}, except regarding cardinal functions, for which we refer to \cite{juhasz}. Although our definition of realcompactness differs from the one given by Engelking, these two definitions are equivalent by \cite[Theorem 3.11.3]{engelking}. For a more general treatment of $\Exp$ and some history, we refer to the articles \cite[Section 8]{van_douwen} and \cite{mrowka}, from which this notation was taken. However, the terminology ``realcompactness number'' is new.

Given a space $X$, make the following definitions:
\begin{itemize}
\item $\kc(X)=\mini\{|\KK|:\KK\text{ is a cover of }X\text{ consisting of compact sets}\}$,
\item $\kcs(X)=\kc(\beta X\setminus X)$.
\end{itemize}
It is clear that $\kcs$ is related to several well-studied notions as follows:
\begin{itemize}
\item $\kcs(X)=0$ iff $X$ is compact,
\item $\kcs(X)=1$ iff $X$ is locally compact but not compact,
\item $\kcs(X)=\omega$ iff $X$ is \v{C}ech-complete but not locally compact.
\end{itemize}
The following simple but very convenient result shows that the value of $\kcs$ does not depend on the choice of compactification.
\begin{proposition}\label{proposition_kcs}
Let $X$ be a space, and let $\gamma X$ be a compactification of $X$. Then $\kcs(X)=\kc(\gamma X\setminus X)$.
\end{proposition}
\begin{proof}
Using \cite[Theorem 3.5.7 and Corollary 3.6.6]{engelking}, it is possible to obtain a perfect surjection $f:\beta X\setminus X\longrightarrow\gamma X\setminus X$. Since $f$ and $f^{-1}$ both preserve compactness (see \cite[Theorem 3.7.2]{engelking}), the desired result easily follows.
\end{proof}

A space is \emph{crowded} if it is non-empty and it has no isolated points. We will use $\RRR$ to denote the space of real numbers with its standard topology, and $\QQQ$ to denote the subspace of $\RRR$ consisting of the rational numbers. We will write $X\approx Y$ to mean that the spaces $X$ and $Y$ are homeomorphic. Given a space $X$, recall that the \emph{Cantor-Bendixson derivative} of $X$ is defined as follows for every ordinal $\xi$:
\begin{itemize}
\item $X^{(0)}=X$,
\item $X^{(\xi+1)}=X^{(\xi)}\setminus\{x\in X^{(\xi)}:x\text{ is isolated in }X^{(\xi)}\}$,
\item $X^{(\xi)}=\bigcap_{\xi'<\xi}X^{(\xi')}$, if $\xi$ is a limit ordinal.
\end{itemize}
Define the \emph{perfect kernel} of $X$ as $\kernel(X)=X^{(\xi)}$, where $\xi$ is the smallest ordinal such that $X^{(\xi)}=X^{(\xi+1)}$. A space is \emph{scattered} if $\kernel(X)=\varnothing$.

The \emph{weight} of a space $X$, denoted by $\weight(X)$, is the maximum between $\omega$ and the smallest cardinal $\kappa$ such that there exists a base for $X$ of size $\kappa$. Given a space $X$ and $x\in X$, define the \emph{character} of $X$ at $x$, denoted by $\chi(x,X)$, to be the maximum between $\omega$ and the smallest cardinal $\kappa$ such that there exists a local base for $X$ at $x$ of size $\kappa$. Given a space $X$ and $S\subseteq X$, define the \emph{pseudocharacter} of $X$ at $S$, denoted by $\psi(S,X)$, to be the maximum between $\omega$ and the smallest cardinal $\kappa$ such that there exists a collection $\OO$ of size $\kappa$ consisting of open subsets of $X$ such that $\bigcap\OO=S$. When $S=\{x\}$, we will simply speak of the pseudocharacter of $X$ at $x$, and denote it by $\psi(x,X)$.

Our reference for set-theoretic notions is \cite{kunen}. For convenience, we will assume that all partial orders\footnote{\,In \cite{kunen}, the discussion is carried out in the more general context of \emph{forcing posets} (that is, $p\leq q$ and $q\leq p$ need not imply $p=q$). Such a level of generality will not be needed here.} have a maximum, which we denote by $\one$. Assume that a partial order $\PPP$ is given. A subset $\CC$ of $\PPP$ is \emph{centered} if for every $F\in [\CC]^{<\omega}$ there exists $q\in\PPP$ such that $q\leq p$ for every $p\in F$. Recall that $\PPP$ is \emph{$\sigma$-centered} if there exist centered subsets $\CC_n$ of $\PPP$ for $n\in\omega$ such that $\PPP=\bigcup_{n\in\omega}\CC_n$. We will say that $\FF\subseteq\PPP$ is a \emph{filter} on $\PPP$ if the following conditions hold:
\begin{itemize}
\item $\one\in\FF$,
\item If $p\in\FF$ and $p\leq q\in\PPP$ then $q\in\FF$,
\item If $p,q\in\FF$ then there exists $r\in\FF$ such that $r\leq p$ and $r\leq q$.
\end{itemize}
An \emph{ultrafilter} is a filter that is maximal with respect to inclusion. It is easy to see that, when $\PPP=\BBB\setminus\{\zero\}$ for some boolean algebra $\BBB$, a filter $\FF$ on $\PPP$ is an ultrafilter iff for every $a\in\PPP$ either $a\in\FF$ of $a^c\in\FF$.

Assume that a countably infinite set $\Omega$ is given. A subset $\CC$ of $[\Omega]^\omega$ has the \emph{strong finite intersection property} (briefly, the $\SFIP$) if $C_0\cap\cdots\cap C_n$ is infinite whenever $n\in\omega$ and $C_0,\ldots,C_n\in\CC$. We will write $A\subseteq^\ast B$ to mean that $A\setminus B$ is finite. We will say that $A\subseteq\Omega$ is a \emph{pseudointersection} of $\CC\subseteq [\Omega]^\omega$ if $A$ is infinite and $A\subseteq^\ast C$ for all $C\in\Aa$. We will denote by $\pppp$ the \emph{pseudointersection number}, that is the minimum cardinality of a subset of $[\omega]^\omega$ with the $\SFIP$ and no pseudointersection. We will denote by $\dddd$ the \emph{dominating number}, that is the minimum cardinality of a subset $\DD$ of $\omega^\omega$ such that for every $f\in\omega^\omega$ there exists $g\in\DD$ such that $f(n)\leq g(n)$ for all but finitely many $n\in\omega$. We will say that $\Aa\subseteq [\Omega]^\omega$ is an \emph{independent family} on $\Omega$ if $A_0\cap\cdots\cap A_m\cap (\omega\setminus B_0)\cap\cdots\cap (\omega\setminus B_n)$ is infinite whenever $m,n\in\omega$ and $A_0,\ldots ,A_m,B_0,\ldots ,B_n\in\Aa$ are distinct. If the set $\Omega$ is not mentioned, we will assume that $\Omega=\omega$. The following result is credited to Nyikos by Matveev in \cite{matveev}.

\begin{proposition}[Nyikos]\label{proposition_nyikos}
There exists an independent family of size $\pppp$ with no pseudointersection.
\end{proposition}
\begin{proof}
Clearly, it will be enough to construct such an independent family on $\omega\times\omega$. Fix an independent family $\Aa$ of size $\pppp$ (see \cite[Exercise III.1.35]{kunen}), and a subset $\CC$ of $[\omega]^\omega$ of size $\pppp$ with the $\SFIP$ and no pseudointersection. Let $\Aa=\{A_\xi:\xi<\pppp\}$ and $\CC=\{C_\xi:\xi<\pppp\}$ be injective enumerations. Set $\Delta^+=\{(m,n)\in\omega\times\omega:m\leq n\}$. It is straightforward to verify that $\{(A_\xi\times C_\xi)\cap\Delta^+:\xi<\pppp\}$ is an independent family on $\omega\times\omega$ of size $\pppp$ with no pseudointersection.
\end{proof}

\section{Calculating the realcompactness number}\label{section_calculating}

The aim of this section is to develop tools to calculate the realcompactness number. We will limit ourselves to the context of Lindel\"{o}f spaces, as the theory is particularly pleasant in this case. Furthermore, since the focus of this article is on countable spaces, this level of generality will clearly suffice. We also remark that Theorem \ref{theorem_Exp} generalizes \cite[Lemma 8.19]{van_douwen} from separable metrizable to Lindel\"{o}f. Corollary \ref{corollary_hechler} gives the simplest non-trivial application of these methods, and it will be a crucial ingredient for the results of Section \ref{section_pseudocharacter}. This theorem first appeared in \cite{hechler_1975}, where it is proved using a result from \cite{hechler_1973}.

\begin{lemma}\label{lemma_lindelof_function}
Let $Z$ be a space, let $X$ be a Lindel\"{o}f subspace of $Z$, and let $K\subseteq Z\setminus X$ be compact. Then there exists a continuous function $f:Z\longrightarrow [0,1]$ such that $f(z)=0$ for every $z\in K$ and $f(z)>0$ for every $z\in X$.
\end{lemma}
\begin{proof}
For every $x\in X$, fix a continuous function $f_x:Z\longrightarrow [0,1]$ such that $f_x(z)=0$ for every $z\in K$ and $f_x(x)=1$. By continuity, for every $x\in X$ there exists an open neighborhood $U_x$ of $x$ in $Z$ such that $f_x(z)>0$ for every $z\in U_x$. Since $X$ is Lindel\"{o}f, there exist $x_n\in X$ for $n\in\omega$ that $X\subseteq\bigcup_{n\in\omega}U_{x_n}$. Finally, define $f:Z\longrightarrow [0,1]$ by setting
$$
f(z)=\sum_{n\in\omega}\frac{f_{x_n}(z)}{2^{n+1}}
$$
for $z\in Z$. It is straightforward to check that $f$ is as desired.
\end{proof}

\begin{theorem}\label{theorem_Exp}
Let $X$ be a Lindel\"{o}f space. Then
$$
\Exp(X)=\maxi\{\weight(X),\kcs(X)\}.
$$
\end{theorem}
\begin{proof}
Set $\kappa=\Exp(X)$ and $\kappa'=\maxi\{\weight(X),\kcs(X)\}$. We begin by showing that $\kappa\leq\kappa'$. By \cite[Theorem 3.5.2]{engelking}, we can fix a compactification $\gamma X$ of $X$ such that $\weight(\gamma X)=\weight(X)$. By Proposition \ref{proposition_kcs}, there exist compact subsets $K_\xi$ of $\gamma X\setminus X$ for $\xi\in\kappa'$ such that $\bigcup_{\xi\in\kappa'}K_\xi=\gamma X\setminus X$. By Lemma \ref{lemma_lindelof_function}, there exist continuous functions $f_\xi:\gamma X\longrightarrow [0,1]$ for $\xi\in\kappa'$ such that $f_\xi(z)=0$ for every $z\in K_\xi$ and $f_\xi(z)>0$ for every $z\in X$. Set $\FF=\{f_\xi:\xi\in\kappa'\}$. Also fix a collection $\GG$ of size at most $\weight(\gamma X)=\weight(X)$ consisting of continuous functions $g:\gamma X\longrightarrow\RRR$ that separates points of $\gamma X$. In other words, whenever $x,y\in\gamma X$ are distinct, there exists $g\in\GG$ such that $g(x)\neq g(y)$. Define $\phi:\gamma X\longrightarrow\RRR^{\FF\cup\GG}$ by setting $\phi(z)(f)=f(z)$ for $z\in\gamma X$ and $f\in\FF\cup\GG$. For notational convenience, we will identify $\RRR^{\FF\cup\GG}$ with $\RRR^\FF\times\RRR^\GG$ in the obvious way.

Notice that $\phi$ is continuous because each one of its coordinates is continuous. Furthermore, our choice of $\GG$ will guarantee that $\phi$ is injective, hence $\phi$ is an embedding by compactness. Therefore
$$
X\approx\phi[X]=\phi[\gamma X]\cap\left((0,\infty)^\FF\times\RRR^\GG\right).
$$
Since $(0,\infty)^\FF\times\RRR^\GG\approx\RRR^{\FF\cup\GG}$ and $\phi[\gamma X]$ is compact, it follows that $\phi[X]\approx X$ is homeomorphic to a closed subspace of $\RRR^{\FF\cup\GG}$. But $|\FF\cup\GG|\leq\kappa'$ by construction, hence $\kappa\leq\kappa'$.

To finish the proof, we will show that $\kappa'\leq\kappa$. Assume without loss of generality that $X$ is a closed subspace of $(0,1)^\kappa$, and let $Z=\cl(X)$, where the closure is taken in $[0,1]^\kappa$. Observe that $Z$ is a compactification of $X$. Denote by $\pi_\xi:[0,1]^\kappa\longrightarrow [0,1]$ for $\xi\in\kappa$ the natural projection on the $\xi$-th coordinate. Since $X$ is closed in $(0,1)^\kappa$, for every $z\in Z\setminus X$ there exists $\xi\in\kappa$ such that $z(\xi)\in\{0,1\}$. In other words, 
$$
Z\setminus X=\bigcup_{\xi\in\kappa}(\pi_\xi^{-1}[\{0,1\}]\cap Z).
$$
But each $\pi_\xi^{-1}[\{0,1\}]\cap Z$ is compact, hence $\kcs(X)\leq\kappa$ by Proposition \ref{proposition_kcs}. Since clearly $\weight(X)\leq\weight\big((0,1)^\kappa\big)=\kappa$, it follows that $\kappa'\leq\kappa$, as desired.
\end{proof}

\begin{corollary}[Hechler]\label{corollary_hechler}
$\Exp(\QQQ)=\dddd$.
\end{corollary}
\begin{proof}
By \cite[Exercise 6.2.A.(e)]{engelking}, it is possible to fix a compactification $\gamma\QQQ$ of $\QQQ$ such that $\gamma\QQQ\approx 2^\omega$. It follows from \cite[Theorem 4.3.23 and Exercise 6.2.A.(b)]{engelking} that $\gamma\QQQ\setminus\QQQ\approx\omega^\omega$. Since $\kc(\omega^\omega)=\dddd$ (see \cite[Theorem 8.2]{van_douwen}), this shows that $\kcs(\QQQ)=\dddd$ by Proposition \ref{proposition_kcs}. An application of Theorem \ref{theorem_Exp} concludes the proof.
\end{proof}

We conclude this section with a result which will enable us to find a bound on the realcompactness number by finding a suitable finite decomposition of the space in question. It will be used in the proofs of Theorem \ref{theorem_example_general} and Corollary \ref{corollary_hereditarily_lindelof}.

\begin{proposition}\label{proposition_finite_decomposition}
Let $X$ be a Lindel\"{o}f space. Assume that $n\in\omega$ and $X_0,\ldots,X_n$ are Lindel\"{o}f subspaces of $X$ such that $X=X_0\cup\cdots\cup X_n$. Then
$$
\Exp(X)\leq\maxi\{\Exp(X_0),\ldots,\Exp(X_n),\weight(X)\}.
$$
\end{proposition}
\begin{proof}
Set $\kappa=\maxi\{\Exp(X_0),\ldots,\Exp(X_n),\weight(X)\}$. By \cite[Theorem 3.5.2]{engelking}, we can fix a compactification $\gamma X$ of $X$ such that $\weight(\gamma X)=\weight(X)$. Throughout this proof, we will use $\cl$ to denote closure in $\gamma X$.

Since each $X_i$ is Lindel\"{o}f, by Theorem \ref{theorem_Exp} it is possible to find compact sets $K_{i,\xi}$ for $0\leq i\leq n$ and $\xi\in\kappa$ such that
$$
\cl(X_i)\setminus X_i=\bigcup_{\xi\in\kappa}K_{i,\xi}
$$
for each $i$. On the other hand, using the compactness of each $\cl(X_i)$ and the fact that $\weight(\gamma X)\leq\kappa$, it is easy to see that each $\psi(\cl(X_i),\gamma X)\leq\kappa$. This means that there exist open subsets $U_{i,\xi}$ of $\gamma X_i$ for $0\leq i\leq n$ and $\xi\in\kappa$ such that 
$$
\gamma X\setminus\cl(X_i)=\bigcup_{\xi\in\kappa}(\gamma X\setminus U_{i,\xi})
$$
for each $i$. In conclusion, set
$$
\KK_i=\{K_{i,\xi}:\xi\in\kappa\}\cup\{\gamma X\setminus U_{i,\xi}:\xi\in\kappa\}
$$
for $0\leq i\leq n$, and observe that each $\KK_i$ is a cover of $\gamma X\setminus X_i$ of size at most $\kappa$ consisting of compact sets. Therefore, the sets of the form $K_0\cap\cdots\cap K_n$, where each $K_i\in\KK_i$, will cover $\gamma X\setminus X$. Since there are at most $\kappa$ sets of this form, and they are clearly compact, a further application of Theorem \ref{theorem_Exp} will conclude the proof.
\end{proof}

\section{The lower bound}\label{section_lower}

In this section, we will give a lower bound on $\Exp(X)$ for a countable crowded space $X$ (see Corollary \ref{corollary_lower_bound}). Chronologically, this is the first result that we obtained, and it motivated us to look for the examples constructed in the later sections.

\begin{theorem}\label{theorem_lower_bound_omega}
Let $\kappa<\pppp$ be an infinite cardinal, and let $X$ be a countable crowded subspace of $\omega^\kappa$. Then $X$ is not closed in $\omega^\kappa$.
\end{theorem}
\begin{proof}
Define
$$
\PPP=\{x\re a:x\in X\text{ and }a\in [\kappa]^{<\omega}\}.
$$
Given $s,t\in\PPP$, declare $s\leq t$ if $s\supseteq t$, and observe that $\PPP$ is a partial order. Also notice that $\PPP$ is $\sigma$-centered because
$$
\PPP=\bigcup_{x\in X}\{x\re a:a\in [\kappa]^{<\omega}\}.
$$

Given $x\in X$ and $a\in [\kappa]^{<\omega}$, define
\begin{itemize}
\item $D_x=\{s\in\PPP:s(\xi)\neq x(\xi)\text{ for some }\xi\in\dom(s)\}$,
\item $D_a=\{s\in\PPP:s=x\re b\text{ for some }x\in X\text{ and }b\in [\kappa]^{<\omega}\text{ such that }b\supseteq a\}$.
\end{itemize}
Using the fact that $X$ is crowded, one sees that each $D_x$ is dense in $\PPP$. On the other hand, it is trivial to check that each $D_a$ is dense in $\PPP$. Since $\kappa<\pppp$, Bell's Theorem (see \cite[Theorem III.3.61]{kunen})\footnote{\,This result was essentially obtained in \cite{bell}. However, the modern statement given in \cite{kunen} is the one that is needed here.} guarantees the existence of a filter $G$ on $\PPP$ that meets all of these dense sets. It is easy to realize that $\bigcup G\in\cl(X)\setminus X$, where $\cl$ denotes closure in $\omega^\kappa$.
\end{proof}

\begin{corollary}\label{corollary_lower_bound_polish}
Let $\kappa<\pppp$ be an infinite cardinal, and let $Z_\xi$ for $\xi\in\kappa$ be Polish spaces. Set $Z=\prod_{\xi\in\kappa}Z_\xi$. If $X$ is a countable crowded subspace of $Z$ then $X$ is not closed in $Z$.
\end{corollary}
\begin{proof}
Let $X$ be a countable crowded subspace of $Z$. Assume, in order to get a contradiction, that $X$ is closed in $Z$. Notice that the existence of $X$ implies that $|Z_\xi|\geq 2$ for infinitely many values of $\xi$. Therefore, by taking suitable countably infinite products, we can assume without loss of generality, that each $Z_\xi$ is crowded. Given $\xi\in\kappa$, denote by $\pi_\xi:Z\longrightarrow Z_\xi$ the natural projection on the $\xi$-th coordinate. By \cite[Exercise 7.4.17]{engelking}, it is possible to obtain a zero-dimensional dense $\mathsf{G}_\delta$ subspace $Z'_\xi$ of $Z_\xi$ for each $\xi$ such that each $\pi_\xi[X]\subseteq Z'_\xi$. By removing a countable dense subset of $Z'_\xi$ disjoint from $\pi_\xi[X]$, we can assume that each $Z'_\xi$ has no non-empty compact open subsets. It follows from \cite[Theorem 4.3.23 and Exercise 6.2.A.(b)]{engelking} that $Z'_\xi\approx\omega^\omega$ for each $\xi$. Finally, the fact that $X$ is a closed subset of
$$
\prod_{\xi\in\kappa}Z'_\xi\approx (\omega^\omega)^\kappa\approx\omega^\kappa
$$
contradicts Theorem \ref{theorem_lower_bound_omega}.
\end{proof}

\begin{corollary}\label{corollary_lower_bound}
If $X$ is a countable crowded space then $\Exp(X)\geq\pppp$.
\end{corollary}

\section{A countable crowded space with realcompactness number $\pppp$}\label{section_p}

In this section, we will construct the space promised by Theorem \ref{theorem_main} in the case $\kappa=\pppp$. Once this is established, the general case will follow without much trouble (see Section \ref{section_bumping}). We will assume some familiarity with the basic theory of boolean algebras and Stone duality (see \cite{koppelberg}, or \cite[Section III.4]{kunen} for a concise exposition). While the core of the construction is given in the proof of Theorem \ref{theorem_example_p}, we will need several technical preliminaries concerning embeddings and filters.

Let partial orders $\PPP$ and $\PPP'$ be given. Inspired by \cite[Definition III.3.65]{kunen}, we will say that $i:\PPP\longrightarrow\PPP'$ is a \emph{pleasant embedding}\footnote{\,By adding the requirement that $i[A]$ is a maximal antichain in $\PPP'$ whenever $A$ is a maximal antichain in $\PPP$, one obtains the definition of \emph{complete embedding}.} if the following conditions are satisfied:
\begin{enumerate}
\item $i(\mathbb{1})=\mathbb{1}$,
\item $\forall p,q\in\PPP\,\big(p\leq q\rightarrow i(p)\leq i(q)\big)$,
\item $\forall p,q\in\PPP\,\big(p\perp q\leftrightarrow i(p)\perp i(q)\big)$.
\end{enumerate}
We will say that $i$ is a \emph{dense embedding} if it satisfies all of the above conditions plus the following:
\begin{enumerate}
\item[(4)] $i[\PPP]$ is dense in $\PPP'$.
\end{enumerate}
Also recall that $\PPP$ is \emph{separative} if for all $p,q\in\PPP$ such that $p\nleq q$ there exists $r\in\PPP$ such that $r\leq p$ and $r\perp q$.

We will say that $\PPP$ is \emph{meet-friendly} if whenever $p,q\in\PPP$ are compatible the set $\{p,q\}$ has a greatest lower bound, which we will denote by $p\wedge q$. In this case, it is clear that $p\wedge q\in\FF$ whenever $\FF$ is a filter on $\PPP$ and $p,q\in\FF$. Notice that $\PPP$ is meet-friendly iff every centered finite subset $\{p_0,\ldots,p_n\}$ of $\PPP$ has a greatest lower bound, which we will denote by $p_0\wedge\cdots\wedge p_n$. An important example of meet-friendly partial-order is given by $\BBB\setminus\{\zero\}$ whenever $\BBB$ is a boolean algebra. When $\PPP$ is meet-friendly and $\CC$ is a non-empty centered subset of $\PPP$, it makes sense to consider the filter $\FF$ generated by $\CC$ (that is, the smallest filter on $\PPP$ containing $\CC$). In fact, it is easy to check that this filter has the following familiar form:
$$
\FF=\{p\in\PPP:p_0\wedge\cdots\wedge p_n\leq p\text{ for some }n\in\omega\text{ and }p_0,\ldots,p_n\in\CC\}.
$$
Finally, when $\PPP$ and $\PPP'$ are meet-friendly, we will say that a pleasant embedding $i:\PPP\longrightarrow\PPP'$ is \emph{meet-preserving} if the following condition is satisfied:
\begin{enumerate}
\item[(5)] $\forall p,q\in\PPP\,\big(p\not\perp q\rightarrow i(p\wedge q)=i(p)\wedge i(q)\big)$.
\end{enumerate}

\begin{lemma}\label{lemma_ultrafilter}
Let $\PPP$ be a meet-friendly partial order, and let $\FF$ be a filter on $\PPP$. Then the following conditions are equivalent:
\begin{enumerate}
\item[(A)] $\FF$ is an ultrafilter,
\item[(B)] $\forall p\in\PPP\setminus\FF\,\exists q\in\FF\, (p\perp q)$.
\end{enumerate}
\end{lemma}
\begin{proof}
In order to prove that $(\text{A})\rightarrow (\text{B})$, assume that $p\in\PPP\setminus\FF$ is such that $p\not\perp q$ for every $q\in\FF$. Since $\PPP$ is meet-friendly, it makes sense to consider
$$
\GG=\FF\cup\{r\in\PPP:r\geq p\wedge q\text{ for some }q\in\FF\}.
$$
It is straightforward to check that $\GG\supsetneq\FF$ is a filter on $\PPP$, hence $\FF$ is not an ultrafilter. In order to prove that $(\text{B})\rightarrow (\text{A})$, assume that $\GG\supsetneq\FF$ is a filter on $\PPP$. It is clear that any $p\in\GG\setminus\FF$ will witness the failure of condition $(\text{B})$.
\end{proof}

\begin{lemma}\label{lemma_pull}
Let $\PPP$ and $\PPP'$ be meet-friendly partial orders, and let $i:\PPP\longrightarrow\PPP'$ be a meet-preserving pleasant embedding. If $\GG$ is a filter on $\PPP'$ then $i^{-1}[\GG]$ is a filter on $\PPP$.
\end{lemma}
\begin{proof}
Let $\GG$ be a filter on $\PPP'$, and set $\FF=i^{-1}[\GG]$. The fact that $\mathbb{1}\in\FF$ follows from condition $(1)$. In order to show that $\FF$ is upward-closed, let $p\in\FF$ and $q\in\PPP$ be such that $p\leq q$. Since $i(p)\in\GG$ and $i(p)\leq i(q)$ by condition $(2)$, we must have $i(q)\in\GG$. Hence $q\in\FF$, as desired. To complete the proof, pick $p,q\in\FF$. Since $i(p),i(q)\in\GG$, we must have $i(p)\not\perp i(q)$, hence $p\not\perp q$ by condition $(3)$. Therefore, we can consider $p\wedge q$, and observe that $i(p\wedge q)=i(p)\wedge i(q)\in\GG$ by condition $(5)$. It follows that $p\wedge q\in\FF$, as desired.
\end{proof}

\begin{lemma}\label{lemma_push}
Let $\PPP$ be a meet-friendly partial order, let $\BBB$ be a boolean algebra, and let $i:\PPP\longrightarrow\BBB\setminus\{\zero\}$ be a pleasant embedding. Assume that $i[\PPP]$ generates $\BBB$ as a boolean algebra. If $\UU$ is an ultrafilter on $\PPP$ then $i[\UU]$ generates an ultrafilter on $\BBB\setminus\{\zero\}$.
\end{lemma}
\begin{proof}
Pick an ultrafilter $\UU$ on $\PPP$. Since $i[\UU]$ is centered by condition $(2)$, we can consider the filter $\VV$ on $\BBB$ generated by $i[\UU]$. Define
$$
\BBB'=\{a\in\BBB:i(p)\leq a\text{ or }i(p)\leq a^c\text{ for some }p\in\UU\}.
$$

\noindent\textbf{Claim 1.} $\BBB'$ is a boolean subalgebra of $\BBB$.

\noindent\textit{Proof.} The fact that $\mathbb{1}\in\BBB'$ follows from condition $(1)$. Furthermore, it is clear that $a\in\BBB'$ iff $a^c\in\BBB'$. In order to show that $\BBB'$ is closed under $\vee$, pick $a,b\in\BBB'$. If $i(p)\leq a$ or $i(p)\leq b$ for some $p\in\UU$, then it is clear that $a\vee b\in\BBB'$. So let $p,q\in\UU$ be such that $i(p)\leq a^c$ and $i(q)\leq b^c$. Since $\UU$ is a filter, we can pick $r\in\UU$ such that $r\leq p$ and $r\leq q$. It follows from condition $(2)$ that
$$
i(r)\leq i(p)\wedge i(q)\leq a^c\wedge b^c=(a\vee b)^c,
$$
which shows that $a\vee b\in\BBB'$, as desired. $\blacksquare$

\noindent\textbf{Claim 2.} $i[\PPP]\subseteq\BBB'$.

\noindent\textit{Proof.} Pick $a\in i[\PPP]$, and let $p\in\PPP$ be such that $i(p)=a$. If $p\in\UU$ then clearly $a\in\BBB'$, so assume that $p\notin\UU$. Since $\PPP$ is meet-friendly and $\UU$ is an ultrafilter, Lemma \ref{lemma_ultrafilter} yields $q\in\UU$ such that $p\perp q$. Notice that $i(p)\perp i(q)$ by condition $(3)$, hence $i(q)\leq i(p)^c=a^c$. This shows that $a\in\BBB'$, as desired. $\blacksquare$

Since $i[\PPP]$ generates $\BBB$ as a boolean algebra, it follows from Claims 1 and 2 that $\BBB'=\BBB$. By the definition of $\BBB'$, this shows that $\VV$ is an ultrafilter on $\BBB\setminus\{\zero\}$.
\end{proof}

Given $a,b\in [\omega]^{<\omega}$, we will write $a\preccurlyeq b$ to mean $a\subseteq b$ and $b\setminus a\subseteq\omega\setminus\maxi(a)$. We will also write $a\prec b$ to mean $a\preccurlyeq b$ and $a\neq b$. Given a subset $\CC$ of $[\omega]^\omega$ with the $\SFIP$, define
$$
\PPP(\CC)=\{(a,F):a\in [\omega]^{<\omega}\text{ and }F\in [\CC]^{<\omega}\}.
$$
Order $\PPP(\CC)$ by declaring $(a,F)\leq (b,G)$ if the following conditions hold:
\begin{itemize}
\item $b\preccurlyeq a$,
\item $G\subseteq F$,
\item $a\setminus b\subseteq\bigcap G$.	
\end{itemize}
This is of course the standard partial order that generically produces a pseudointersection of $\CC$. We remark that $\PP(\CC)$ is always  meet-friendly. In fact, given $(a,F),(b,G)\in\PPP(\CC)$ such that $(a,F)\not\perp (b,G)$, it is easy to realize that $(a\cup b,F\cup G)$ is the greatest lower bound of $\{(a,F),(b,G)\}$.

\begin{theorem}\label{theorem_example_p}
There exists a countable crowded space $X$ such that $\Exp(X)=\pppp$.	
\end{theorem}
\begin{proof}
By Proposition \ref{proposition_nyikos}, we can fix an independent family $\Aa$ of size $\pppp$ with no pseudointersection. Without loss of generality, assume that for every $n\in\omega$ there exists $A\in\Aa$ such that $n\notin A$. Set $\PPP=\PPP(\Aa)$. Given $a\in [\omega]^{<\omega}$, denote by $\UU_a$ the filter on $\PPP$ generated by $\{(a,F):F\in [\Aa]^{<\omega}\}$.

\noindent\textbf{Claim 1.} Each $\UU_a$ is an ultrafilter on $\PPP$.

\noindent\textit{Proof.} Fix $a\in [\omega]^{<\omega}$. Pick $(b,G)\in\PPP$ that is compatible with $(a,F)$ for every $F\in [\Aa]^{<\omega}$. We will show that $(b,G)\in\UU_a$. It is easy to realize that either $a\prec b$ or $b\preccurlyeq a$. First assume that $a\prec b$, and pick $n\in b\setminus a$. By our choice of $\Aa$, there exists $A\in\Aa$ such that $n\notin A$. It is clear that $(a,\{A\})\perp (b,G)$, which shows that this case is impossible. Therefore $b\preccurlyeq a$. Pick $(c,H)\in\PPP$ such that $(c,H)\leq (b,G)$ and $(c,H)\leq (a,G)$. Observe that $a\setminus b\subseteq c\setminus b\subseteq\bigcap G$, hence $(a,G)\leq (b,G)$. It follows that $(b,G)\in\UU_a$, as desired. $\blacksquare$

\noindent\textbf{Claim 2.} $\PPP$ is separative.

\noindent\textit{Proof.} Pick $(a,F),(b,G)\in\PPP$ such that $(a,F)\nleq (b,G)$. We will find $(c,H)\leq (a,F)$ that is incompatible with $(b,G)$. First assume that $b\nsubseteq a$. Pick $n\in\bigcap F$ big enough so that $n>\maxi(a)$ and $n>\maxi(b)$. It is easy to realize that $(a\cup\{n\},F)\leq (a,F)$ and that $(a\cup\{n\},F)\perp (b,G)$.

Next, assume that $b\subseteq a$ but $b\not\preccurlyeq a$. This means that there exists $n\in a\setminus b$ such that $n<\maxi(b)$, hence $(a,F)\perp (b,G)$. Next, assume that $b\preccurlyeq a$ but there exists $n\in a\setminus b$ such that $n\notin\bigcap G$. It is clear that $(a,F)\perp (b,G)$ in this case as well.

Finally, assume that $b\preccurlyeq a$, $a\setminus b\subseteq\bigcap G$, but $G\nsubseteq F$. Let $A\in G\setminus F$. Since $\Aa$ is an independent family, it is possible to pick $n\in (\omega\setminus A)\cap\bigcap F$ big enough so that $n>\maxi(a)$. It is clear that $(a\cup\{n\},F)\leq (a,F)$ and that $(a\cup\{n\},F)\perp (b,G)$. $\blacksquare$

Given $p\in\PPP$, we will use the notation $p\down=\{q\in\PPP :q\leq p\}$. Declare $U\subseteq\PPP$ to be open if $p\down\subseteq U$ for every $p\in U$. We will denote by $\RO(\PPP)$ the regular open algebra of $\PPP$ according to this topology (see \cite[Definition III.4.6]{kunen}). Define $i:\PPP\longrightarrow\RO(\PPP)\setminus\{\zero\}$ by setting $i(p)=p\down$ for $p\in\PPP$. It follows from Claim 2 and \cite[Lemma III.4.8 and Exercise III.4.15]{kunen} that $i$ is a well-defined dense embedding, and it is straightforward to verify that $i$ is meet-preserving. Furthermore, by \cite[Exercise III.4.11]{kunen}, the following stronger form of condition $(2)$ holds:
\begin{enumerate}
\item[$(2')$] $\forall p,q\in\PPP\,\big(p\leq q\leftrightarrow i(p)\leq i(q)\big)$.
\end{enumerate}

Let $\BBB$ be the boolean subalgebra of $\RO(\PPP)$ generated by $i[\PPP]$, and let
$$
Z=\{\VV:\VV\text{ is an ultrafilter on }\BBB\setminus\{\zero\}\}
$$
denote the Stone space of $\BBB$. Given $b\in\BBB$, we will denote by $[b]=\{\VV\in Z:b\in\VV\}$ the corresponding basic clopen subset of $Z$. By Claim 1 and Lemma \ref{lemma_push}, each $i[\UU_a]$ generates an ultrafilter on $\BBB\setminus\{\zero\}$, which we will denote by $\VV_a$. Finally, set
$$
X=\{\VV_a:a\in [\omega]^{<\omega}\}.
$$

\noindent\textbf{Claim 3.} $Z$ is crowded.

\noindent\textit{Proof.} This is equivalent to showing that $\BBB$ has no atoms. So pick $b\in\BBB\setminus\{\zero\}$. By condition $(4)$, there exists $p\in\PPP$ such that $i(p)\leq b$. Now choose any $q\in\PPP$ such that $q<p$, and observe that $\zero <i(q)<i(p)\leq b$ by condition $(2')$. This shows that $b$ is not an atom. $\blacksquare$

\noindent\textbf{Claim 4.} $X$ is a countable dense subset of $Z$.

\noindent\textit{Proof.} The fact that $X$ is countable is clear. In order to see that $X$ is dense, pick a non-empty open subset $U$ of $Z$. Without loss of generality, assume that $U=[b]$ for some $b\in\BBB\setminus\{\zero\}$. By condition $(4)$, we can actually assume that $b=(a,F)\down$ for some $(a,F)\in\PPP$. It is clear that $b\in i[\UU_a]\subseteq\VV_a$, hence $\VV_a\in [b]=U$. $\blacksquare$

It follows from Claims 3 and 4 that $X$ is a countable crowded space, and that $Z$ is a compactification of $X$. Furthermore, it is easy to see that $\weight(X)\leq\weight(Z)=|\BBB|=\pppp$, where $|\BBB|\geq\pppp$ holds by condition $(2')$. Since $\Exp(X)\geq\pppp$ by Corollary \ref{corollary_lower_bound}, the following claim will conclude the proof by Theorem \ref{theorem_Exp}.

Fix an enumeration $\Aa=\{A_\xi:\xi\in\pppp\}$. Given $\xi\in\pppp$, set
$$
U_\xi=\bigcup_{a\in [\omega]^{<\omega}}[(a,A_\xi)\down],
$$
and observe that each $U_\xi$ is an open subset of $Z$.

\noindent\textbf{Claim 5.} $X=\bigcap_{\xi\in\pppp}U_\xi$.

\noindent\textit{Proof.} The inclusion $\subseteq$ is straightforward. In order to prove the inclusion $\supseteq$, pick $\VV\in\bigcap_{\xi\in\pppp}U_\xi$. This means that for every $\xi\in\pppp$ we can fix $a_\xi\in [\omega]^{<\omega}$ such that $(a_\xi,A_\xi)\down\in\VV$. Set $\UU=i^{-1}[\VV]$, and observe that $\UU$ is a filter on $\PPP$ by Lemma \ref{lemma_pull}. Furthermore, it is clear that each $(a_\xi,A_\xi)\in\UU$.

Set $a=\bigcup_{\xi\in\pppp}a_\xi$. First assume, in order to get a contradiction, that $a$ is infinite. Using the fact that each $(a_\xi,A_\xi)\in\UU$, it is easy to verify that $a\setminus\maxi(a_\xi)\subseteq A_\xi$ for each $\xi$. In other words, the set $a$ is a pseudointersection of $\Aa$, which is a contradiction.

Therefore $a$ is finite, hence we can fix $\xi\in\pppp$ such that $a=a_\xi$. We will show that $\UU=\UU_a$, which easily implies that $\VV=\VV_a$, thus concluding the proof. Since $\UU_a$ is an ultrafilter, it will be enough to show that $\UU_a\subseteq\UU$. So pick $(a,F)\in\UU_a$, where $F=\{A_{\xi_0},\ldots,A_{\xi_k}\}$. As in the proof that $\PPP$ is meet-friendly, one sees that
$$
(a,F\cup\{A_\xi\})=(a_{\xi_0},A_{\xi_0})\wedge\cdots\wedge (a_{\xi_k},A_{\xi_k})\wedge (a_{\xi},A_{\xi})\in\UU,
$$
which clearly implies $(a,F)\in\UU$, as desired. $\blacksquare$
\end{proof}

\section{Bumping-up the weight}\label{section_bumping}

In this section, we will finally exhibit the examples promised in Theorem \ref{theorem_main}. The strategy is to start with the space given by Theorem \ref{theorem_example_p}, then artificially increase its weight. Since we will achieve this by adding a single point, Proposition \ref{proposition_finite_decomposition} will guarantee that the realcompactness number will not grow more than we want it to.

\begin{theorem}\label{theorem_example_general}
Let $\kappa$ be a cardinal such that $\pppp\leq\kappa\leq\cccc$. Then there exists a countable crowded space $X_\kappa$ such that $\Exp(X_\kappa)=\kappa$.
\end{theorem}
\begin{proof}
By Theorem \ref{theorem_example_p}, we can fix a countable crowded space $X$ with $\Exp(X)=\pppp$. Since $X$ is zero-dimensional, Lindel\"{o}f, and non-compact, it is possible to fix non-empty clopen subsets $V_n$ of $X$ for $n\in\omega$ such that $\bigcup_{n\in\omega}V_n=X$ and $V_m\cap V_n=\varnothing$ whenever $m\neq n$. Also fix an independent family $\Aa$ of size $\kappa$ (see \cite[Exercise III.1.35]{kunen}), and set
$$
\FF=\{F\subseteq\omega:A_0\cap\cdots\cap A_k\subseteq^\ast F\text{ for some }k\in\omega\text{ and }A_0,\ldots,A_k\in\Aa\}.
$$

Define the space $X_\kappa$ by taking $X\cup\{\ast\}$ as the underlying set, where $\ast\notin X$ simply denotes an extra point, and by declaring $U\subseteq X\cup\{\ast\}$ open exactly when one of the following conditions holds:
\begin{itemize}
\item $U$ is an open subset of $X$,
\item $U=\{\ast\}\cup U'$ for some open subset $U'$ of $X$ such that $\bigcup\{V_n:n\in F\}\subseteq U'$ for some $F\in\FF$.
\end{itemize}
It is straightforward to check that this topology is regular, hence zero-dimensional by \cite[Corollary 6.2.8]{engelking}, hence Tychonoff. Notice that every point of $X$ is non-isolated in $X_\kappa$ because $X$ is crowded, while $\ast$ is non-isolated in $X_\kappa$ because $\varnothing\notin\FF$. This means that $X_\kappa$ is crowded.

\noindent\textbf{Claim.} $\chi(\ast,X_\kappa)=\kappa$.

\noindent\textit{Proof.} It is clear that the open sets of the form
$$
\{\ast\}\cup\bigcup\{V_n:n\in (A_0\cap\cdots\cap A_k)\setminus\ell\},
$$
where $k,\ell\in\omega$ and $A_0,\ldots,A_k\in\Aa$, constitute a local base for $X_\kappa$ at $\ast$. Therefore $\chi(\ast,X_\kappa)\leq\kappa$. Now assume, in order to get a contradiction, that there exists a local base $\BB$ for $X_\kappa$ at $\ast$ such that $|\BB|<\kappa$, and assume without loss of generality that every element of $\BB$ is in the form described above. Since $|\Aa|=\kappa$, we can fix $A\in\Aa$ that does not appear in any of these descriptions of the elements of $\BB$. Since $\BB$ is a local base, there must be $k,\ell\in\omega$ and $A_0,\cdots,A_k\in\Aa\setminus\{A\}$ such that
$$
\{\ast\}\cup\bigcup\{V_n:n\in (A_0\cap\cdots\cap A_k)\setminus\ell\}\subseteq \{\ast\}\cup\bigcup\{V_n:n\in A\}.
$$
It follows that $(A_0\cap\cdots\cap A_k)\setminus\ell\subseteq A$, which contradicts the fact that $\Aa$ is an independent family. $\blacksquare$

Using the above claim and the fact that $\weight(X)\leq\Exp(X)=\pppp\leq\kappa$, one sees that $\weight(X_\kappa)=\kappa$, which clearly implies $\Exp(X_\kappa)\geq\kappa$. To conclude the proof, simply apply Proposition \ref{proposition_finite_decomposition} with $X_0=X$ and $X_1=\{\ast\}$.
\end{proof}

\section{Scattered spaces}\label{section_scattered}

The purpose of this section is to show that the realcompactness number of a countable space is essentially determined by its perfect kernel (see Corollary \ref{corollary_hereditarily_lindelof}). In order to achieve this, we will obtain a bound for the pseudocharacter of a scattered space in any compactification (see Theorem \ref{theorem_scattered}).

We begin with a technical lemma, which was inspired by \cite[Theorem 3.9.2]{engelking}. Recall that an \emph{extension} of a space $X$ is a space $Z$ in which $X$ is dense.

\begin{lemma}\label{lemma_scattered}
Let $X$ be a space, let $\kappa$ be an infinite cardinal, and let $\CC_\xi$ for $\xi\in\kappa$ be open covers of $X$. Assume that $\bigcap\DD\neq\varnothing$ whenever $\DD$ satisfies the following requirements:
\begin{enumerate}
\item $\DD$ consists of closed subsets of $X$,
\item $D_0\cap\cdots\cap D_n\neq\varnothing$ whenever $n\in\omega$ and $D_0,\ldots,D_n\in\DD$,
\item $\forall\xi\in\kappa\,\exists D\in\DD\,\exists U\in\CC_\xi\, (D\subseteq U)$.
\end{enumerate}
Then $\psi(X,Z)\leq\kappa$ for every extension $Z$ of $X$.
\end{lemma}
\begin{proof}
Pick an extension $Z$ of $X$. Define
$$
U_\xi=\bigcup\{U\text{ open in }Z:U\cap X\in\CC_\xi\}
$$
for $\xi\in\kappa$, and observe that each $U_\xi$ is open in $Z$. Therefore, the following claim will conclude the proof.

\noindent\textbf{Claim.} $X=\bigcap_{\xi\in\kappa}U_\xi$.

\noindent\textit{Proof.} The inclusion $\subseteq$ is clear, since each $\CC_\xi$ is an open cover of $X$. In order to prove the inclusion $\supseteq$, pick $x\in\bigcap_{\xi\in\kappa}U_\xi$. Define
$$
\DD=\{\cl(U)\cap X:U\text{ is an open neighborhood of }x\text{ in }Z\},
$$
where $\cl$ denotes closure in $Z$, and observe that $\bigcap\DD\subseteq\{x\}$. It is obvious that condition $(1)$ holds. Using the fact that $X$ is dense in $Z$, one sees that condition $(2)$ holds. Finally, using the assumption that $x\in\bigcap_{\xi\in\kappa}U_\xi$, it is straightforward to verify that condition $(3)$ holds. Therefore $\bigcap\DD\neq\varnothing$, which implies $\bigcap\DD=\{x\}$. Since clearly $\bigcap\DD\subseteq X$, it follows that $x\in X$. $\blacksquare$
\end{proof}

Given a scattered space $X$ and $S\subseteq X$, define the \emph{Cantor-Bendixson rank} of $S$ as
$$
\rank(S)=\mini\{\xi:X^{(\xi)}\cap S=\varnothing\}.
$$
When $S=\{x\}$, we will simply write $\rank(x)$ to mean $\rank(S)$.

\begin{theorem}\label{theorem_scattered}
Let $X$ be a scattered space, and let $\kappa=\weight(X)$. Then $\psi(X,Z)\leq\kappa$ for every extension $Z$ of $X$.
\end{theorem}
\begin{proof}
Assume without loss of generality that $X$ is non-empty. Define a subset $S$ of $X$ to be \emph{focused} if there exists a unique $x\in S$ such that $\rank(x)=\rank(S)$. Notice that $\rank(x)$ is a successor ordinal for every $x\in X$. Given $x\in X$ such that $\rank(x)=\xi+1$, set
$$
U_x=\big(X\setminus X^{(\xi)}\big)\cup\{x\}.
$$
It is not hard to realize that each $U_x$ is a focused open neighborhood of $x$. Furthermore, it is clear that every neighborhood of $x$ contained in $U_x$ will still be focused. Using this fact, it is possible to fix a base $\BB$ for $X$ of size at most $\kappa$ such that $\BB$ consists of focused open sets. Then define
$$
\CC_0=\{U\in\BB:\cl(U)\subseteq V\text{ for some }V\in\BB\text{ with }\rank(V)=\rank(U)\}.
$$

\noindent\textbf{Claim 1.} $\CC_0$ is a base for $X$ consisting of focused open sets.

\noindent\textit{Proof.} Pick $x\in X$ and an open neighborhood $U$ of $x$. Since $\BB$ is a base, there exists $V\in\BB$ such that $x\in V\subseteq U\cap U_x$. By regularity, there exists $U'\in\BB$ such that $x\in U'\subseteq\cl(U')\subseteq V$. It is clear that $U'\in\CC_0$, as desired. $\blacksquare$

Given $U_0,U_1,U_2,U_3\in\BB$ such that $\cl(U_3)\subseteq U_2\subseteq U_1\subseteq\cl(U_1)\subseteq U_0$, define
$$
\CC_{U_0,U_1,U_2,U_3}=\{U_2\}\cup\{V\in\CC_0:V\subseteq U_0\setminus\cl(U_3)\}\cup\{X\setminus\cl(U_1)\},
$$
and observe that each $\CC_{U_0,U_1,U_2,U_3}$ is an open cover of $X$. Observe that there exists at least one such cover because $X$ is non-empty, and let $\CC_\xi$ for $\xi\in\kappa\setminus\{0\}$ enumerate them all (possibly with repetitions). By Lemma \ref{lemma_scattered}, it will be enough to show that $\bigcap\DD\neq\varnothing$ whenever $\DD$ satisfies conditions $(1)$, $(2)$ and $(3)$. So pick such a $\DD$.

Without loss of generality, assume that $\DD$ is closed under finite intersections. Set $\xi=\mini\{\rank(D):D\in\DD\}$. By considering $\DD'=\{D'\cap D:D\in\DD\}$ instead of $\DD$, where $D'\in\DD$ is a fixed element of rank $\xi$, we can also assume that $\rank(D)=\xi$ for every $D\in\DD$. By condition $(3)$ applied to the cover $\CC_0$, we can define
$$
\xi'=\mini\{\rank(U):U\in\CC_0\text{ and }D\subseteq U\text{ for some }D\in\DD\}.
$$

\noindent\textbf{Claim 2.} $\xi'=\xi$.

\noindent\textit{Proof.} It is easy to realize that $\xi\leq\xi'$. Now assume, in order to get a contradiction, that $\xi<\xi'$. Pick $U_1\in\CC_0$ and $D\in\DD$ such that $\rank(U_1)=\xi'$ and $D\subseteq U_1$. By the definition of $\CC_0$, we can also fix $U_0\in\BB$ such that $\cl(U_1)\subseteq U_0$ and $\rank(U_0)=\rank(U_1)$. Since $\CC_0$ consists of focused sets, there exists a unique $x\in U_1$ such that $\rank(x)=\xi'$. Notice that $x$ is also the unique element of $U_0$ whose rank is $\xi'$. The fact that $\rank(D)=\xi<\xi'=\rank(x)$ shows that $x\notin D$. Therefore, by Claim 1, it is possible to find $U_2\in\CC_0$ such that $x\in U_2\subseteq\cl(U_2)\subseteq U_1\setminus D$. By regularity, we can also fix $U_3\in\BB$ such that $x\in U_3\subseteq\cl(U_3)\subseteq U_2$. 

By condition $(3)$, there exist $V\in\CC_{U_0,U_1,U_2,U_3}$ and $D'\in\DD$ such that $D'\subseteq V$. It follows from condition $(2)$ that $V\neq U_2$ and $V\neq X\setminus\cl(U_1)$, hence $V\in\CC_0$ and $V\subseteq U_0\setminus\cl(U_3)$. It follows that $\rank(V)\geq\xi'$ by minimality, while on the other hand $\rank(V)\leq\rank(U_0)=\xi'$. In conclusion, we have shown that $\rank(V)=\xi'$. Since $V\in\CC_0$ is a focused set, it follows that there exists $y\in V$ such that $\rank(y)=\xi'$. This is a contradiction, since $x$ is the unique element of $U_0$ of rank $\xi'$. $\blacksquare$

By Claim 2, we can fix $U\in\CC_0$ and $D\in\DD$ such that $\rank(U)=\xi$ and $D\subseteq U$. Let $x$ be the unique element of $U$ of rank $\xi$. Clearly, the following claim will conclude the proof.

\noindent\textbf{Claim 3.} $x\in\bigcap\DD$.

\noindent\textit{Proof.} Pick any $D'\in\DD$. Observe that $D'\cap D\in\DD$ because $\DD$ is closed under finite intersections, hence $\rank(D'\cap D)=\xi$. Since $D\cap D'\subseteq D\subseteq U$, it follows that $x\in D\cap D'\subseteq D'$, as desired. $\blacksquare$
\end{proof}

We remark that the proof of Theorem \ref{theorem_scattered} would be more pleasant under the assumption that $X$ is zero-dimensional. However, it is not true that every scattered space is zero-dimensional (see \cite{solomon}).

\begin{corollary}\label{corollary_scattered}
Let $X$ be a Lindel\"{o}f scattered space. Then
$$
\Exp(X)=\weight(X).
$$
\end{corollary}
\begin{proof}
This follows from Theorems \ref{theorem_Exp} and \ref{theorem_scattered}, by considering a compactification of $X$.	
\end{proof}

\begin{corollary}\label{corollary_hereditarily_lindelof}
Let $X$ be a hereditarily Lindel\"{o}f space. Then
$$
\Exp(X)=\maxi\{\Exp(\kernel(X)),\weight(X)\}.
$$
\end{corollary}
\begin{proof}
Set $X_0=\kernel(X)$ and $X_1=X\setminus X_0$. The inequality $\geq$ is clear, since $X_0$ is a closed subspace of $X$. In order to prove the inequality $\leq$, observe that $X_1$ is a scattered Lindel\"{o}f space, hence $\Exp(X_1)=\weight(X_1)$ by Corollary \ref{corollary_scattered}. The desired result then follows from Proposition \ref{proposition_finite_decomposition}.
\end{proof}

\section{The pseudocharacter of countable subsets of products}\label{section_pseudocharacter}

The main purpose of this section is to give an upper bound for $\Exp(X)$ when $X$ is a countable space (see Corollary \ref{corollary_upper_bound}). This will follow from a general analysis of the pseudocharacter of countable subsets of products of first-countable spaces. As a by-product, we will also obtain improved versions of some of the examples given by Theorem \ref{theorem_example_general} (see Corollary \ref{corollary_topological_group}).

\begin{lemma}\label{lemma_pseudocharacter}
Let $Z$ be a first-countable space. Then $\psi(X,Z)\leq\dddd$ for every countable subset $X$ of $Z$.
\end{lemma}
\begin{proof}
Pick a countable subset $X$ of $Z$. The desired result is trivial if $X$ is empty, so assume that $X$ is non-empty and let $X=\{x_n:n\in\omega\}$ be an enumeration. Given $n\in\omega$, fix a local base $\{U_{n,m}:m\in\omega\}$ for $Z$ at $x_n$ such that $U_{n,0}\supseteq U_{n,1}\supseteq\cdots$. Also fix $\{f_\xi:\xi<\dddd\}\subseteq\omega^\omega$ such that for every $f:\omega\longrightarrow\omega$ there exists $\xi<\dddd$ such that $f(n)\leq f_\xi(n)$ for every $n\in\omega$. Set
$$
U_\xi=\bigcup_{n\in\omega}U_{n,f_\xi(n)}
$$
for $\xi<\dddd$, and observe that each $U_\xi$ is an open subset of $Z$. It is straightforward to verify that $X=\bigcap_{\xi<\dddd}U_\xi$, which concludes the proof.
\end{proof}

\begin{theorem}\label{theorem_pseudocharacter}
Let $\kappa$ be an infinite cardinal, and let $Z_\xi$ for $\xi\in\kappa$ be first-countable spaces. Set $Z=\prod_{\xi\in\kappa}Z_\xi$. Then $\psi(X,Z)\leq\maxi\{\dddd,\kappa\}$ for every countable subset $X$ of $Z$.
\end{theorem}
\begin{proof}
We will proceed by transfinite induction on $\kappa$. The case $\kappa=\omega$ is given by Lemma \ref{lemma_pseudocharacter}. Now assume that $\kappa$ is an uncountable cardinal and that the desired result holds for all infinite cardinals below $\kappa$. Pick a countable subset $X$ of $Z$. Fix $\Omega_\xi\subseteq\kappa$ for $\xi\in\kappa$ so that the following conditions are satisfied:

\begin{enumerate}
\item $x\re\Omega_0\neq x'\re\Omega_0$ whenever $x,x'\in X$ are distinct,
\item $\Omega_\xi\subseteq\Omega_{\xi'}$ whenever $\xi\leq\xi'$,
\item $|\Omega_\xi|<\kappa$ for each $\xi$,
\item $\bigcup_{\xi\in\kappa}\Omega_\xi=\kappa$.
\end{enumerate}

Given $\Omega\subseteq\kappa$, set $Z(\Omega)=\prod_{\xi\in\Omega}Z_\xi$. Given $\xi\in\kappa$, denote by $\pi_\xi:Z\longrightarrow Z(\Omega_\xi)$ the natural projection, which is of course obtained by setting $\pi_\xi(z)=z\re\Omega_\xi$. By condition $(3)$ and the inductive hypothesis, for every $\xi\in\kappa$ we can fix a collection $\OO_\xi$ consisting of open subsets of $Z(\Omega_\xi)$ such that $|\OO_\xi|\leq\maxi\{\dddd,\kappa\}$ and $\bigcap\OO_\xi=\pi_\xi[X]$. Now define
$$
\OO'_\xi=\{U\times Z(\kappa\setminus\Omega_\xi):U\in\OO_\xi\}
$$
for $\xi\in\kappa$, where we identify $Z$ with $Z(\Omega_\xi)\times Z(\kappa\setminus\Omega_\xi)$ in the obvious way. Observe that each $|\OO'_\xi|\leq\maxi\{\dddd,\kappa\}$. Therefore, the following claim will conclude the proof.

\noindent\textbf{Claim.} $\bigcap_{\xi\in\kappa}\bigcap\OO'_\xi=X$.

\noindent\textit{Proof.} The inclusion $\supseteq$ is trivial. In order to prove the inclusion $\subseteq$, pick $z$ that belongs to the left-hand side. It is easy to realize that $\pi_\xi(z)\in\pi_\xi[X]$ for each $\xi$. So there exist $x_\xi\in X$ for $\xi\in\kappa$ such that each $\pi_\xi(x_\xi)=\pi_\xi(z)$. However, using conditions $(1)$ and $(2)$, one sees that there must be $x\in X$ such that each $x_\xi=x$. Finally, by condition $(4)$, it is clear that $z=x\in X$, as desired. $\blacksquare$
\end{proof}

\begin{corollary}\label{corollary_upper_bound}
Let $X$ be a countable space. Then $\Exp(X)\leq\cccc$.
\end{corollary}
\begin{proof}
Observe that $\weight(X)\leq\cccc$ because $X$ is countable. Therefore, since $X$ is zero-dimensional by \cite[Corollary 6.2.8]{engelking}, we can assume that $X$ is a subspace of $2^\cccc$ by \cite[Theorem 6.2.16]{engelking}. Observe that $\cl(X)$ is a compactification of $X$, where $\cl$ denotes closure in $2^\cccc$. Therefore, by Theorem \ref{theorem_pseudocharacter}, there exist open subsets $U_\xi$ of $2^\cccc$ for $\xi\in\cccc$ such that $X=\bigcap_{\xi\in\cccc}U_\xi$. It follows that
$$
\cl(X)\setminus X=\bigcup_{\xi\in\cccc}(\cl(X)\setminus U_\xi),
$$
which shows that $\kcs(X)\leq\cccc$ by Proposition \ref{proposition_kcs}. In conclusion, an application of Theorem \ref{theorem_Exp} yields the desired result.
\end{proof}

\begin{corollary}\label{corollary_exact_pseudocharacter}
Let $\kappa\geq\dddd$ be a cardinal, and let $Z_\xi$ for $\xi\in\kappa$ be first-countable spaces such that each $|Z_\xi|\geq 2$. Set $Z=\prod_{\xi\in\kappa}Z_\xi$. Then $\psi(X,Z)=\kappa$ for every non-empty countable subset $X$ of $Z$.
\end{corollary}
\begin{proof}
Pick a countable subset $X$ of $Z$. First observe that $\psi(X,Z)\leq\kappa$ by Theorem~\ref{theorem_pseudocharacter}. Now assume, in order to get a contradiction, that $\psi(X,Z)<\kappa$. Let $\OO$ be a collection of open subsets of $Z$ such that $|\OO|<\kappa$ and $\bigcap\OO=X$. Fix $x\in X$, then define
$$
\OO'=\OO\cup\{Z\setminus\{z\}:z\in X\setminus\{x\}\}.
$$
It is clear that $|\OO'|<\kappa$ and $\bigcap\OO'=\{x\}$. Since each $|Z_\xi|\geq 2$, this contradicts the fact that $2^\kappa$ has pseudocharacter $\kappa$ at each point (see \cite[5.3.(b)]{juhasz}).
\end{proof}

Before stating the next two results, we remind the reader that a product of at most $\cccc$ separable spaces is separable (see \cite[Corollary 2.3.16]{engelking}).

\begin{corollary}\label{corollary_dense_subspace}
Let $\kappa$ be a cardinal such that $\dddd\leq\kappa\leq\cccc$, and let $Z_\xi$ for $\xi\in\kappa$ be compact separable first-countable spaces such that each $|Z_\xi|\geq 2$. Set $Z=\prod_{\xi\in\kappa}Z_\xi$. If $X$ is a countable dense subspace of $Z$ then $\Exp(X)=\kappa$.
\end{corollary}
\begin{proof}
Pick a countable dense subspace $X$ of $Z$. The desired result follows from Theorem \ref{theorem_Exp} and Corollary \ref{corollary_exact_pseudocharacter}, since $Z$ is a compactification of $X$.
\end{proof}

\begin{corollary}\label{corollary_topological_group}
Let $\kappa$ be a cardinal such that $\dddd\leq\kappa\leq\cccc$. Then there exists a countable (crowded) topological group $X$ such that $\Exp(X)=\kappa$.
\end{corollary}
\begin{proof}
Pick a countable dense subset $D$ of $2^\kappa$, and let $X$ be the subgroup of $2^\kappa$ generated by $D$. It follows from Corollary \ref{corollary_dense_subspace} that $\Exp(X)=\kappa$.
\end{proof}

\section{Open questions}

Our first three questions aim at ``improving'' the examples given in Sections \ref{section_p} and \ref{section_bumping}. Recall that a space $X$ is \emph{extremally disconnected} if the closure of every open subset of $X$ is open.

\begin{question}
For which cardinals $\kappa$ such that $\pppp\leq\kappa\leq\cccc$ does there exist an extremally disconnected countable crowded space $X$ such that $\Exp(X)=\kappa$?
\end{question}

We do not know whether the assumption that $\dddd\leq\kappa\leq\cccc$ in Corollary \ref{corollary_topological_group} can be weakened to $\pppp\leq\kappa\leq\cccc$. This is the content of the next question. 

\begin{question}\label{question_topological_group}
For which cardinals $\kappa$ such that $\pppp\leq\kappa\leq\cccc$ does there exist a countable (crowded) topological group $X$ such that $\Exp(X)=\kappa$?
\end{question}

Recall that a space $X$ is \emph{homogeneous} if for all $x,y\in X$ there exists a homeomorphism $h:X\longrightarrow X$ such that $h(x)=y$. Since every topological group is clearly homogeneous, the following might be a first step towards answering Question \ref{question_topological_group}.

\begin{question}
For which cardinals $\kappa$ such that $\pppp\leq\kappa\leq\cccc$ does there exist a countable (crowded) homogeneous space $X$ such that $\Exp(X)=\kappa$?
\end{question}

Finally, it would be desirable to ``complete the picture'' regarding the results of Section \ref{section_pseudocharacter}. Notice that, by Corollary \ref{corollary_exact_pseudocharacter}, it only remains to answer Question \ref{question_pseudocharacter} in the case when $\kappa<\dddd$.

\begin{question}\label{question_pseudocharacter}
Assume that $\kappa$ is a given cardinal. What are the possible values of $\psi(X,Z)$ when $Z$ is a product of $\kappa$ first-countable spaces and $X$ is a countable subset of $Z$?
\end{question}

\subsection*{Acknowledgements}
The first-listed author acknowledges the support of the FWF grant P 35655. The second-listed author acknowledges the support of the FWF grants P 35655 and P 35588. The third-listed author acknowledges the support of the FWF grant I 5930.

\end{document}